\documentclass{amsart}

\calclayout

\usepackage{enumitem}
\setenumerate{nolistsep}
\setitemize{nolistsep}
\usepackage[square, numbers, comma, sort&compress]{natbib}
\usepackage[final]{pdfpages}
\usepackage{graphicx}
\usepackage{esint}
\usepackage{latexsym}
\usepackage{hyperref}
\usepackage{cleveref}
\usepackage{autonum}

\usepackage{amsmath}
\usepackage{amssymb}

\usepackage{tikz}
\usetikzlibrary{graphs, angles, quotes,patterns,hobby}
\usepackage{pgfplots}
\pgfplotsset{compat=1.6}

\usepackage{lipsum}

\newcommand\blfootnote[1]{%
  \begingroup
  \renewcommand\thefootnote{}\footnote{#1}%
  \addtocounter{footnote}{-1}%
  \endgroup
}

\renewcommand{\Re}{\operatorname{Re}}
\renewcommand{\Im}{\operatorname{Im}}

\DeclareMathOperator{\dist}{dist}

\DeclareMathOperator{\supp}{supp}

\DeclareMathOperator{\pv}{p.v}
\DeclareMathOperator{\C}{\mathcal{C}}

\theoremstyle{plain}
\newtheorem{theor}{Theorem}

\theoremstyle{plain}
\newtheorem*{theor*}{Theorem}

\theoremstyle{plain}
\newtheorem{prop}{Proposition}

\theoremstyle{remark}
\newtheorem{rem}{Remark}

\theoremstyle{definition}
\newtheorem{defin}{Definition}

\theoremstyle{plain}

\theoremstyle{remark}

\theoremstyle{plain}
\newtheorem{lemm}{Lemma}

\numberwithin{equation}{section}

\usepackage{hyperref}
\hypersetup{
    colorlinks=true,
    linkcolor=blue,
    filecolor=blue,      
    urlcolor=blue,
    citecolor=red,
}

\begin{document}


\title{Measures that define a compact Cauchy transform}


\author{Carmelo Puliatti}
\address{BGSMath and Departament de Matematiques, Universitat Autonoma
de Barcelona, 08193, Bellaterra, Barcelona, Catalonia}
\email{puliatti@mat.uab.cat}

\begin{abstract}
The aim of this work is to provide a geometric characterization of the positive Radon measures $\mu$ with compact support on the plane such that the associated Cauchy transform defines a compact operator from $L^2(\mu)$ to $L^2(\mu).$ It turns out that a crucial role is played by the density of the measure and by its Menger curvature.
\end{abstract}
 \maketitle
 
\section{Introduction}
\blfootnote{\textit{2010 Mathematics Subject Classification.} 42B20, 28A80.\\
Partiallly supported by MTM-2016-77635-P,  MDM-2014-044 (MICINN, Spain), 2017-SGR-395 (Catalonia), and by Marie Curie ITN MAnET (FP7-607647).}
In what follows we will identify the plane with the complex field $\mathbb{C}.$ Let $\mu$ be a positive Radon measure on $\mathbb{C}$ with compact support and without atoms. For $\epsilon>0$, $f\in L^1_{loc}(\mu)$ and $z\in\mathbb{C}$ we set
\begin{equation}
\C_{\epsilon,\mu}f(z):=\int_{|z-w|>\epsilon}\frac{f(w)}{z-w}d\mu(w).
\end{equation}
We define the Cauchy transform operator $\C_{\mu}$ in a principal value sense, i.e., as the limit
\begin{equation}
\C_{\mu}f(z):=\lim_{\epsilon\to 0}\C_{\epsilon,\mu}f(z)
\end{equation}
for every $z$ such that the above limit exists. We say that the Cauchy transform is bounded from $L^2(\mu)$ to $L^2(\mu)$ if the truncated operators $\C_{\epsilon,\mu}:L^2(\mu)\to L^2(\mu)$ are bounded uniformly in $\epsilon$. 
As a consequence of the work of Mattila and Verdera (see \cite{mattila_verdera_jems} or the book by Tolsa \cite[Chapter 8]{tolsa_book}), the Cauchy transform is bounded from $L^2(\mu)$ to $L^2(\mu)$ if and only if the truncated operators $\{\C_{\epsilon,\mu}\}_\epsilon$ converge as $\epsilon$ tends to $0$ in the weak operator topology of the space of bounded linear operators from $L^2(\mu)$ to $L^2(\mu)$. Moreover, if we denote as $\C_\mu^w$ the limit of the aforementioned net, for all $f\in L^2(\mu)$ and for almost all $z$, the principal value $\C_\mu f(z)$ exists and it coincides with $\C_\mu^w f(z)$. This is a peculiarity of the Cauchy transform and it does not hold for every singular integral operator.
Now, it makes sense to introduce the following definition. 
\begin{defin}\label{def_compact_cauchy}We say that the Cauchy transform is compact from $L^2(\mu)$ to $L^2(\mu)$ if it is bounded in $L^2(\mu)$ and $\C^w_\mu$ is compact as an operator from $L^2(\mu)$ to $L^2(\mu)$.
\end{defin}
As a consequence of the results we cited, one may replace $\C^w_\mu$ in Definition \ref{def_compact_cauchy} with the principal value $\C_\mu$.
A useful tool to study the Cauchy transform of a measure $\mu$ is the so-called Menger curvature $c(\mu)$, that was first related to the Cauchy transform in \cite{melnikov} and \cite{MV}. 
Denoting by $R(z,w,\zeta)$ the radius of the circumference passing though $z,w$ and $\zeta$, and defining
\begin{equation}
 c^2_{\mu}(z):=\int\int\frac{1}{R(z,w,\zeta)^2}d\mu(w)d\mu(\zeta),
\end{equation}
the Menger curvature of $\mu$ is defined as
\begin{equation}
c^2(\mu):=\int c^2_\mu(z)d\mu(z).
\end{equation}
Let $d,n\in\mathbb{N}$ with $n\leq d.$ Given a cube $Q$ in $\mathbb{R}^d,$ we denote by $l(Q)$ its side length and by
\begin{equation}\label{density_cu}
\Theta^n_\mu(Q):=\frac{\mu(Q)}{l(Q)^n}
\end{equation}
its $n-$dimensional density. If $z\in\mathbb{R}^d,$ we define the upper density of $\mu$ at $z$ as
\begin{equation}\label{sup_dens}
\Theta^{n,*}_{\mu}(z):=\limsup_{l(Q)\to 0}\Theta^n_\mu(Q),
\end{equation}
where $Q$ spans over the cubes centered at $z$. Replacing the superior limit with the inferior limit we get the definition of the lower density $\Theta^n_{*,\mu}(z).$ If $\Theta^{n,*}_\mu(z)=\Theta^n_{*,\mu}(z),$ we denote that common value as $\Theta^n_{\mu}(z)$ and call it "density of $\mu$ at the point $z$". In the case $d=2$ and $n=1$, for brevity we write $\Theta_\mu(Q):=\Theta^1_\mu(Q)$ and we omit the index $n$ from the notation for the upper and lower densities at any point.\\
The aim of the present work is to characterize the measures $\mu$ on the plane such that its associated Cauchy transform defines a compact operator from $L^2(\mu)$ into $L^2(\mu).$ Not much literature is available concerning compactness for Singular Integral Operators in the context of Euclidean spaces equipped with a measure different from the Lebesgue measure. We point out that a $T(1)-$like criterion for the compactness of Calder\'on-Zygmund operators in Euclidean spaces is available due to the work of Villarroya \cite{villa}.\\
We denote by $K(L^2(\mu),L^2(\mu))$ the space of compact operators from $L^2(\mu)$ to $L^2(\mu)$. 
We will see that a crucial condition to get a compact Cauchy transform is to require that
\begin{equation}
\Theta^*_\mu(z)=0
\end{equation}
for every $z\in\mathbb{C}.$ 
Our main result is the following.
\begin{theor}\label{main_thm}
Let $\mu$ be a compactly supported positive Radon measure on $\mathbb{C}$ without atoms. The following conditions are equivalent:
\begin{enumerate}[label=(\alph*)]
\item \label{cond_a} $\C_\mu$ is compact from $L^2(\mu)$ to $L^2(\mu)$. 
\item the two following properties hold:\label{cubes}
\begin{enumerate}[label=(\arabic*)]
\item $\Theta^*_{\mu}(z)=0$ uniformly, which means that the limit in \eqref{sup_dens} is $0$ uniformly in $z\in\mathbb{C}$. \label{thm_dens}
\item ${c^2(\mu\llcorner Q)}/{\mu(Q)}\to 0$ as $l(Q)\to 0,$ where $\mu\llcorner Q$ stands for the restriction of $\mu$ to the cube $Q$.\label{thm_curv}
\end{enumerate}
\item the truncated operators $\C_{\epsilon,\mu}$ converge as $\epsilon\to 0$ in the operator norm of the space of bounded linear operators from $L^2(\mu)$ to $L^2(\mu)$.\label{cond_c}
\end{enumerate}
\end{theor}
We remark that the proof of the theorem relies on the $T(1)$-theorem for the Cauchy transform  (see \cite{tolsa_book}) and that one could replace the cubes with balls in condition \ref{cubes}, as well as in \eqref{density_cu}.\\
Theorem \ref{main_thm} can be generalized to higher dimensions taking into consideration the ${n}$-Riesz transform $\mathcal{R}^n_\mu$ on $\mathbb{R}^d$ for $n\leq d$ in place of the Cauchy transform. If $\mu$ is a compactly supported positive Radon measure on $\mathbb{R}^d$ without atoms, $\epsilon>0$, $f\in L^1_{loc}(\mu)$ and $z\in\mathbb{R}^d,$ the truncated Riesz transform is defined as
\begin{equation}
\mathcal{R}^{n}_{\epsilon,\mu}f(z):=\int_{|z-w|>\epsilon}\frac{x-y}{|x-y|^{n+1}}f(y)d\mu(y).
\end{equation}
As in the case of the Cauchy transform, thanks again to the result in \cite{mattila_verdera_jems}, the weak limit $\mathcal{R}^{n,w}_\mu$ of $\mathcal{R}^n_{\epsilon,\mu}$ as $\epsilon\to 0$ exists provided the $\mathcal{R}^n_{\epsilon,\mu}$ are uniformly bounded on $L^2(\mu),$ and we can understand the compactness of the Riesz transform as in Definition \ref{def_compact_cauchy}. The main difference with the Cauchy transform is that the only case in which boundedness is known to imply that the principal value exists is for $n=d-1$. This is a consequence of \cite{NTV}.\\	
In this more general context, Theorem \ref{main_thm} reads as follows.
\begin{theor}\label{den_thm_gen}
Let $\mu$ be a compactly supported positive Radon measure on $\mathbb{R}^d$ without atoms.
The following conditions are equivalent:
\begin{enumerate}[label=(\alph*)]
\item $\mathcal{R}^n_\mu$ is compact from $L^2(\mu)$ to $L^2(\mu)$.
\item the two following properties hold:\label{cond_b_gen}
\begin{enumerate}[label=(\arabic*)]
\item\label{dens_cond_gen} $\Theta^{n-1,*}_{\mu}(z)=0$ uniformly in $z\in\mathbb{R}^d.$
\item ${||\mathcal{R}^n_\mu \chi_Q||^2_{L^2(\mu\llcorner Q)}}/{\mu(Q)}\to 0$ as $l(Q)\to 0.$ 
\end{enumerate}
\item the truncated operators $\mathcal{R}^n_{\epsilon,\mu}$ converge as $\epsilon\to 0$ in the operator norm of the space of bounded linear operators from $L^2(\mu)$ to $L^2(\mu)$.
\end{enumerate}
\end{theor}
Theorem \ref{den_thm_gen} can be proved with minor changes of the proof that we will discuss for the case of the Cauchy transform. Combining condition \ref{cond_b_gen} in Theorem \ref{den_thm_gen} with \cite[Theorem 1.6]{mattila_verdera_jems}, we can infer that if $\mathcal{R}^n_\mu$ is compact then the principal value $\mathcal{R}^n_{\mu}(x)$ exists for $\mu-$almost every $x$.\\
The work is structured as follows.
In Section \ref{int_sec} we deal with two toy models: first we show a direct proof of the non-compactness of the Cauchy transform of the one dimensional Lebesgue measure on a segment. Then, we prove that the Cauchy transform of a disc endowed with the planar Lebesgue measure is compact. In Section \ref{proof_sec} we prove Theorem \ref{main_thm}. As an application of this result, Section \ref{cantor_sec} is devoted to the discussion of the case of the general planar Cantor sets. We conclude the exposition with a remark on the generalization of the main theorem to other Singular Integral Operators.

\subsection*{Notation.} Throughout this work, we use the standard notations $A\lesssim B$ if  there exits an absolute positive constant $C$ such that $A\leq C B$. If $A\lesssim B$ and $B\lesssim A,$ we write $A\approx B.$\\
Given an operator $T:X\to Y,$ we use the notation $||T||_{X\to Y}$ for its operator norm.

\section{The Cauchy transform on a segment and on the disc}\label{int_sec}
It may be worth recalling the following property of compact operators: if $X$ and $Y$ are Banach spaces, $T:X\to Y$ is a compact operator and $\{u_k\}_k$ is a sequence in $X$ such that $u_k\rightharpoonup u$ for some $u\in X$ (weak convergence), then $Tu_k\to Tu$ (strongly) in $Y.$ We will use this property both for the proof of the following proposition and for the proof of the main theorem.
Let us start by considering the Cauchy transform on a segment. Given an interval $I$ on the real line, we denote by $\mathcal{H}^1$ the $1-$dimensional Hausdorff measure and use the notation $L^2(I):=L^2(\mathcal{H}^1\llcorner(I\times \{0\})).$ Without loss of generality, we analyze the case $I=[0,1]$.
\begin{prop}
Let $\mu:=\mathcal{H}^1\llcorner([0,1]\times \{0\}).$ The Cauchy transform $\mathcal{C}_\mu$ is not a compact operator from $L^2(\mu)$ into $L^2(\mu).$
\end{prop}
\begin{proof}
Let $\C_\mu$ be the Cauchy transform of the measure $\mu:=\mathcal{H}^1\llcorner([0,1]\times \{0\}),$ acting on functions belonging to $L^2([0,1]).$\\
For $k\in\mathbb{N}$, let us define the function $f_k:\mathbb{R}\to\mathbb{R}$ as
\begin{equation}
f_k(x):= 2^{(k-1)/2}\big(\chi_{[1/2-2^{-k},1/2]}(x)-\chi_{[1/2,1/2+2^{-k}]}(x)\big).
\end{equation}
Notice that $||f_k||_{L^2([0,1])}=||f_k||_{L^2(\mathbb{R})}=1$ and that $\{f_k\}_k$ converges to $0$ in the weak topology of $L^2([0,1])$. However, $\{f_k\}_k$ does not converge in the strong topology of $L^2([0,1])$.\\
Let us denote by $Hf_k$ the Hilbert transform of $f_k$
\begin{equation}
Hf_k(x):=\pv\int\frac{f_k(y)}{x-y}dy
\end{equation}
for $x\in\mathbb{R}.$ We claim that $Hf_k$ does not converge to $0$ in the strong topology of $L^2([0,1])$. Hence $\C_\mu=H$ is not compact in $L^2(\mu)$.\\
A well known fact regarding Hilbert transform (see e.g. \cite{Stein}) is that
\begin{equation}\label{hilb_iso}
||Hf||_{L^2(\mathbb{R})}=\pi||f||_{L^2(\mathbb{R})}
\end{equation}
for every $f\in L^2(\mathbb{R}).$\\
The following argument proves that $||\C_\mu f_k||_{L^2([0,1])}=||H f_k||_{L^2([0,1])}$ tends to $\pi$ for $k\to\infty$.\\
It is clearly enough to show that
\begin{equation}\label{dec_1_infty}
||H f_k||^2_{L^2([1,+\infty))}\to 0 \qquad\text{ for }\qquad k\to\infty
\end{equation}
and
\begin{equation}\label{dec_infty_0}
||H f_k||^2_{L^2((-\infty,0])}\to 0 \qquad\text{ for }\qquad k\to\infty.
\end{equation}
To prove \eqref{dec_1_infty}, first notice that for $y\in\supp f_k$ and $x\geq 1$, it holds that $|x-y|\geq |x-3/4|.$ Then
\begin{equation}
\begin{split}
||H f_k||^2_{L^2(1,+\infty)}&=\int^{+\infty}_1\Big|\int^{1/2+2^{-k}}_{1/2-2^{-k}}\frac{f_k(y)}{x-y}dy\Big|^2dx \\
&\leq \int^{+\infty}_1\frac{1}{|x-\frac{3}{4}|^2}\Big(\int^{1/2+2^{-k}}_{1/2-2^{-k}}|f_k(y)|dy\Big)^2dx\\
&\leq 2^{-k+1}\int^{+\infty}_1\frac{1}{|x-\frac{3}{4}|^2}dx\lesssim 2^{-k},
\end{split}
\end{equation}
which gives \eqref{dec_1_infty}. The proof of \eqref{dec_infty_0} is analogous.
 \end{proof}
Now we turn to analyze the Cauchy transform on the disc.
Let $D:=D(0,1)=\{z\in\mathbb{C}:|z|< 1\}$ and let $\epsilon>0.$ Let $\mu=dA$ be the $2-$dimensional Lebesgue measure restricted to $D$.
\begin{lemm}\label{disc_lemma}
The operator $\C_{\epsilon,\mu}:L^2(dA)\to L^2(dA)$ is compact for every $\epsilon>0.$
\end{lemm}
\begin{proof}
Let $z,w\in\mathbb{C}$ and let $K_{\epsilon}(z,w):={\chi_{D(z,\epsilon)^c}(w)}/({z-w}).$
By the Hilbert-Schmidt's Theorem (see \cite[Theorem 6.12]{brezis}), to prove the lemma it is enough to show that the integral
\begin{equation}
\int_{D} |K_{\epsilon}(z,w)|^2 dA(z)
\end{equation}
converges.
This occurs in our case because
\begin{equation}
\int_{D} |K_{\epsilon}(z,w)|^2dA(z)\leq\frac{A(D)}{\epsilon^2}=\frac{\pi}{\epsilon^2},
\end{equation}
and so the proof is complete.
\end{proof}
For $f\in L^2(dA)$ let us define
\begin{equation}
\C^\epsilon_\mu f(z):=\C_\mu f(z)-\C_{\epsilon,\mu}f(z).
\end{equation}
By Lemma \ref{disc_lemma}, to prove that $\C_\mu$ belongs to $K(L^2(dA),L^2(dA))$ it suffices to prove that $||\C^\epsilon_\mu||_{L^2(dA)\to L^2(dA)}\to 0$ as $\epsilon\to 0.$ Indeed, this implies that  $\{\C_{\epsilon,\mu}\}_{\epsilon>0}$ converges in operator norm to the Cauchy transform, which proves that it is compact.\\
For $f\in L^2(dA)$, a direct computation using polar coordinates gives
\begin{align}
\int_D|\C^\epsilon_\mu f(z)|^2dA(z)&=\int_D\Big|\int_{|z-w|<\epsilon}\frac{f(w)}{z-w}dA(w)\Big|^2dA(z)\\
&=\int_D\Big|\int^{2\pi}_0\int^\epsilon_0{e^{-i\theta}}f(z+re^{i\theta})\chi_D(z+re^{i\theta})drd\theta\Big|^2dA(z)\\
&\leq \int_D\int^{2\pi}_0\int^\epsilon_0|f(z+re^{i\theta})\chi_D(z+re^{i\theta})|^2drd\theta dA(z)\\
&\leq 2\pi \epsilon ||f||^2_{L^2(dA)},
\end{align}
where in the last inequality we used Fubini's Theorem. Hence
\begin{equation}
||\C^\epsilon_\mu f||_{L^2(dA)\to L^2(dA)}\leq (2\pi\epsilon)^{1/2}
\end{equation}
and so $\C_{\mu}\in K(L^2(dA),L^2(dA)).$
\begin{rem}
The integral
\begin{equation}\label{obs}
\int_{B(0,1)}\frac{1}{|z|}dA(z)
\end{equation}
plays a crucial role in the proof of the compactness of the Cauchy transform of a disc.
When focusing on the general case in which $dA$ is replaced by a measure $\mu,$ one may be tempted to guess that we need a density condition which gives that the analogue of \eqref{obs} converges. This drives our attention to measures with zero linear density, which we will prove to be a necessary condition for the Cauchy transform to be compact.
\end{rem}
\section{The proof of Theorem \ref{main_thm}} \label{proof_sec}
\subsection{Necessary conditions for the compactness.}
In order to prove the necessity of the conditions in Theorem \ref{main_thm}, we argue by contradiction: assuming that
there exists a sequence of cubes $\{Q_j\}_j$ such that $l(Q_j)\to 0$ but $\limsup\Theta^1_\mu(Q_j)>0,$ we will prove that the Cacuhy transform does not define a compact operator on $L^2(\mu).$ \\
We recall that a necessary condition to have the $L^2(\mu)$-boundedness of $\mathcal{C}_{\mu}$ is that $\mu$ has linear growth (see \cite{david_wavelets}). In particular we choose to denote by $C_0$ a positive constant such that
\begin{equation}\label{growth_condition}
\mu(Q)\leq C_0 l(Q)
\end{equation}
for every cube in $\mathbb{R}^2.$\\
Let us state a technical lemma that we are going to use to prove Theorem \ref{main_thm}. The proof is a variant of Lemma 2.3 in \cite{leger}.
\begin{lemm} \label{lemm_cubes}
Suppose that there is a sequence of cubes $\{Q_j\}$ such that $l(Q_j)\to 0$ and
\begin{equation}
\limsup_j \Theta^1_\mu(Q_j)\equiv \Theta >0.
\end{equation}
Let $Q$ be a cube in $\{Q_j\}_j$ such that $\Theta^1_\mu(Q)\geq \Theta/2$. There exist $C_1,C'_1\in \mathbb{N}$ both bigger than $1$
 such that we can find two cubes $Q'$ and $Q''$ with side length $l(Q)/C_1$ and with the following properties
\begin{enumerate}
\item $\dist(Q',Q'')\approx l(Q')$.
\item $\min (\mu(Q'),\mu(Q''))\geq {l(Q)}/{C'_1}$ . \label{cond_2}
\end{enumerate}
\end{lemm}
\begin{proof}
Let us argue by contradiction. Let us split $Q$ into a grid of $C_1^2$ equal cubes of side length $l(Q)/C_1$ whose sides are parallel to the sides of $Q$; we denote this collection of cubes as $\mathcal{D}.$  Let us assume that each couple of cubes $Q',Q''\in\mathcal{D}$, is such that either they touch (so that $\dist(Q',Q'')=0$) or $\min(\mu(Q'),\mu(Q''))\leq l(Q)/C'_1$.\\
By construction we have that
\begin{equation}\label{eq1}
\sum_{\widetilde{Q}\in\mathcal{D}}\mu(\widetilde{Q})=\mu(Q)=\Theta(Q)l(Q).
\end{equation}
Now let us consider the family
\begin{equation}
\mathcal{G}:=\Big\{\widetilde{Q}\in\mathcal{D}:\mu(\widetilde{Q})\geq \frac{l(Q)}{C'_1}\Big\}.
\end{equation}
By hypothesis, all the cubes in $\mathcal{G}$ must be contained in a single cube of side length $3l(Q)/C_1$ that we denote as $P.$ The growth condition \eqref{growth_condition} gives
\begin{equation}
\mu(P)\leq C_0 l({P})=3 {C_0}l(Q)/C_1,
\end{equation}
so that
\begin{equation}\label{eq2}
\sum_{\widetilde{Q}\in\mathcal{G}}\mu(\widetilde{Q})\leq 3 \frac{C_0}{C_1}l(Q).
\end{equation}
For those cubes of $\mathcal{D}$ not belonging to $\mathcal{G}$ we can write
\begin{equation}\label{eq3}
\sum_{\widetilde{Q}\in\mathcal{D}\setminus\mathcal{G}}\mu(\widetilde{Q})\leq \frac{C_1^2}{C_1'}l(Q).
\end{equation}
By hypothesis we have that $\Theta(Q)\geq \Theta/2.$ Then, gathering \eqref{eq1}, \eqref{eq2} and \eqref{eq3} we get the inequality
\begin{equation}\label{ineq_contradiction}
\frac{C_1^2}{C'_1}+3\frac{C_0}{C_1}\geq \frac{\Theta}{2}.
\end{equation}
Choosing $C_1$ and $C'_1$ big enough, \eqref{ineq_contradiction} gives a contradiction.
\end{proof}
\begin{rem} Using the growth condition for the measure $\mu,$ the condition (\ref{cond_2}) in the statement of Lemma \ref{lemm_cubes} actually implies that $Q'$ and $Q''$ are such that
\begin{equation}\label{eq_masses}
\mu(Q')\approx\mu(Q'')\approx\mu(Q).
\end{equation}
\end{rem}
As a consequence of the proof of Lemma \ref{lemm_cubes}, it is not difficult to see that we can choose $Q'$ and $Q''$ arbitrarily small. This will lead to a contradiction.\\
Given a cube cube $Q$, we define the function $\varphi_Q:={\chi_Q}/{\mu(Q)}^{1/2}.$ We have that $||\varphi_Q||_{L^2(\mu)}=1$ for every cube and that
\begin{equation}
\varphi_{Q_j}\rightharpoonup 0
\end{equation}
weakly in $L^2(\mu)$ for every sequence of cubes $\{Q_j\}_j$ such that $l(Q_j)\to 0.$\\
Now, taking $Q$, $Q'$ and $Q''$ as in Lemma \ref{lemm_cubes}, we can write
\begin{equation}\label{cont_1}
|\langle\C_{\mu}\varphi_{Q'},\varphi_{Q''}\rangle|\leq ||\C_{\mu}\varphi_{Q'}||_{L^2(\mu)}||\varphi_{Q''}||_{L^2(\mu)}=||\C_{\mu}\varphi_{Q'}||_{L^2(\mu)}.
\end{equation}
The proof of the necessity of the density condition of Theorem \ref{main_thm} follows from \eqref{cont_1} if we can prove that $|\langle\C_{\mu}\varphi_{Q'},\varphi_{Q''}\rangle|$ is bounded from below away from $0$; indeed, this would imply that $||\C_{\mu}\varphi_{Q'}||_{L^2(\mu)}$ does not converge to $0$, which contradicts the compactness of the Cauchy transform.
\begin{lemm}Let $Q'$ and $Q''$ be as in Lemma \ref{lemm_cubes}. There exists a constant $c>0,$ independent on the side length of the cubes, such that
\begin{equation}
c \frac{\mu(Q)}{l(Q')} \leq |\langle \C_\mu\varphi_{Q'},\varphi_{Q''}\rangle| .
\end{equation}
\end{lemm}
\begin{proof}Suppose without loss of generality that the centers of the cubes $Q'$ and $Q''$ are aligned with the real axis. By \eqref{eq_masses}, we have that
\begin{equation}\label{e1}
\big|\Re\langle\C_\mu\varphi_{Q'},\varphi_{Q''}\rangle\big|\approx\frac{1}{\mu(Q)}\big|\Re\langle\C_\mu\chi_{Q'},\chi_{Q''}\rangle\big|.
\end{equation}
Suppose that $\Re(z-w)>0$ for every $z\in Q''$ and $w\in Q'.$ Then
\begin{equation}\label{e2}
\big|\Re\langle\C_\mu\chi_{Q'},\chi_{Q''}\rangle\big|=\Big|\Re\int_{Q''}\C_\mu \chi_{Q'}(z)d\mu(z)\Big|=\int_{Q''}\int_{Q'}\frac{\Re(z-w)}{|z-w|^2}d\mu(w)d\mu(z).
\end{equation}
Lemma \ref{lemm_cubes} ensures that, if $z\in Q''$ and $w\in Q',$ we have that $|z-w|\approx l(Q)\approx l(Q'),$ so that, using \eqref{e1},\eqref{e2}
we have
\begin{equation}\label{e3}
|\Re\langle\C_\mu \chi_{Q'},\chi_{Q''}\rangle|\gtrsim \frac{\mu(Q)^2}{l(Q')}.
\end{equation}
The Lemma follows from \eqref{e3} and \eqref{e1}.
\end{proof} 
The following lemma gives a necessary condition for the Cauchy transform of a measure to be compact in terms of the curvature.
\begin{lemm}
Let $\mu$ be a compactly supported positive Radon measure on $\mathbb{C}$ without atoms. Suppose that $\C_{\mu}$ defines a compact operator from $L^2(\mu)$ to $L^2(\mu).$ Then
\begin{equation}\label{cur_cond}
\frac{c^2(\mu\llcorner Q)}{\mu(Q)}\to 0
\end{equation}
as $l(Q)\to 0.$
\end{lemm}
\begin{proof}
Let $Q$ be an arbitrary cube in $\mathbb{R}^2.$
From a formula due to Tolsa and Verdera (see \cite{TV}, Theorem 2) applied to the measure $\mu\llcorner Q,$ we have that
\begin{equation}\label{form_norm}
||\C_{\mu}\chi_Q||^2_{L^2(\mu\llcorner Q)} =\frac{\pi^2}{3}\int_Q \theta_{\mu}(z)^2d\mu(z)+\frac{1}{6}c^2(\mu\llcorner Q).
\end{equation}
Since we suppose $\C$ to be compact, we proved that $\theta_{\mu}(z)=0$ for every $z\in\mathbb{R}^2$, so that the integral in the right hand side of \eqref{form_norm} vanishes.\\
Consider a sequence of cubes $\{Q_j\}_j$ such that $l(Q_j)\to 0$ as $j\to \infty.$ As before, if we define $\varphi_j:={\chi_{Q_j}}/{\mu(Q_j)^{1/2}},$ we have that 
\begin{equation}
\varphi_j\rightharpoonup  0
\end{equation}
weakly in $L^2(\mu).$ Then, since we suppose the Cauchy transform to be compact, we have that 
\begin{equation}
||\C_{\mu}\varphi_j||^2_{L^2(\mu)}\to 0
\end{equation}
for $j\to\infty.$ The inequalities
\begin{equation}
||\C_{\mu}\chi_{Q_j}||^2_{L^2(\mu\llcorner Q_j)}\leq ||\C_{\mu}\chi_{Q_j}||^2_{L^2(\mu)}\leq \mu(Q_j) ||\C_{\mu}\varphi_j||^2_{L^2(\mu)},
\end{equation}
and \eqref{form_norm} conclude the proof of the Lemma.
\end{proof}
\subsection{Sufficient conditions for the compactness.}
The proof that we present now relies on the $T(1)$-theorem of David and Journe. More specifically, we prove that proper truncates of the Cauchy transform are compact operators and, then, we estimate the operator norm of the difference between $\C$ and those truncates.\\
Let $\mu$ be a positive Radon measure with compact support in $\mathbb{C}$. Let $z\in\supp\mu$ and let $Q_z$ be a square containing the support of $\mu$ and centered at $z$. Let $l(Q_z)$ denote its side length. For $j\in\mathbb{N}$ we denote as $Q_j(z)$ the square centered at $z$ and with side-length $2^{-j}l(Q_z)$. Moreover, we define
\begin{equation}
\Delta_j(z):=Q_j(z)\setminus Q_{j+1}(z).
\end{equation}
Exploiting Hilbert-Schmidt's Theorem, a proof analogous to the one of Lemma \ref{disc_lemma} shows that the truncated operator
\begin{equation}
T_jf(z):=\int_{\Delta_j(z)}\frac{f(w)}{z-w}d\mu(w)
\end{equation}
is a compact operator from $L^2(\mu)$ to $L^2(\mu).$ Let us define
\begin{equation}
\C^N_\mu f(w):=\sum_{j=0}^{N-1} T_jf(w)
\end{equation}
and show that, under the hypothesis on the measure reported in the statement of Theorem \ref{main_thm}, it converges in the $L^2(\mu)-L^2(\mu)$ operator norm to the Cauchy transform. This will prove that $\C_{\mu}\in K(L^2(\mu),L^2(\mu)).$\\
The $T(1)-$Theorem (see \cite[Chapter 3]{tolsa_book}) provides the estimate
\begin{equation}\label{T1}
\begin{split}
||\C_\mu-\C^{N-1}_\mu||_{L^2(\mu)\to L^2(\mu)}&\approx \sup_{z\in\supp\mu}\sup_{\widetilde{Q}\subseteq Q_N(z)}\Theta(\widetilde{Q})+\sup_{z\in\supp\mu}\sup_{\widetilde{Q}\subseteq Q_N(z)}\frac{||\C\chi_{\widetilde{Q}}||_{L^2(\mu\llcorner\widetilde{Q})}}{\mu(\widetilde{Q})^{1/2}}\\
&\equiv I_N + II_N.
\end{split}
\end{equation}
First, $I_N\to 0$ as $N\to\infty$ by the hypothesis $(2)$ of Theorem \ref{main_thm} on the density of $\mu$.\\
To show that $II_N\to 0$ as $N\to\infty,$ it suffices to recall formula \eqref{form_norm}, which yields
\begin{equation}
||\C_\mu\chi_{\widetilde{Q}}||^2_{L^2(\mu\llcorner\widetilde{Q})}\lesssim c^2(\mu\llcorner \widetilde{Q}).
\end{equation}
The ratio $c^2(\mu\llcorner \widetilde{Q})/\mu(\widetilde{Q})$ has the correct behavior due to the condition $(2)$ of Theorem \ref{main_thm}. This concludes the proof of the equivalence of the conditions \ref{cond_a} and \ref{cubes}. In order to complete the proof of the theorem, it suffices to observe that the equivalence of \ref{cubes} and \ref{cond_c} follows from \eqref{T1}.

\section{An example: a generalized planar Cantor set}\label{cantor_sec}
As an application of Theorem \ref{main_thm} we analyze the particular case of the planar Cantor sets (see e.g. \cite[p. 87]{garnett}). Let $Q^0:=[0,1]^2$ be the unit square and let $\lambda:=\{\lambda_n\}_{n=1}^\infty$ be a sequence of non-negative numbers such that $0\leq\lambda_n\leq 1/2$ for every $n=1,2,\ldots$. The Cantor set is defined by means of an inductive construction:
\begin{itemize}
\item define $4$ squares $\{Q^1_j\}^4_{j=1}$ of side length $\lambda_1$ such that each one of them contains a distinct vertex of $Q_0$ and call $E_1:=\cup^4_{j=1}Q^1_j.$
\item iterate the first step for each of the $4$ cubes but using $\lambda_2$ as a scaling factor. As a result we get $2^4=16$ squares of side length $\sigma_2=\lambda_1\lambda_2.$ We denote those squares as $\{Q^2_j\}_j.$ Then, define the second-step approximation of the Cantor set as $E_2:=\cup^{2^4}_{j=1}Q^2_j.$
\item as a result of $n$ analogous iterations, at the $n-$th step we get a collection of $4^n$ cubes $\{Q^n_j\}_j$ whose side length is $\sigma_n:=\prod^{n}_{j=1}\lambda_j$ and a set $E_n:=\cup^{4^n}_{j=1}Q^n_j.$
\end{itemize}
The planar Cantor set is defined as
\begin{equation}
E=E(\lambda):=\bigcap^\infty_{n=1}E_n.
\end{equation}
We denote by $p$ the canonical probability measure associated with $E(\lambda).$ In particular, $p$ is uniquely identified by imposing that $p(Q)=4^{-n}$ for every square that composes $E_n.$
We denote by $\C_{p}$ the Cauchy transform associated with the measure $p$.\\
Let $\theta_k:=2^{-k}\sigma_k^{-1}$. It is known (see e.g. \cite{tolsa_book}, Lemma 4.29) that for the probability measure on the Cantor set, it holds that
\begin{equation}
c^2_p(x)\approx \sum_{k=0}^\infty \theta^2_k
\end{equation}
for every $x\in E(\lambda)$.\\
As a consequence of Theorem \ref{main_thm}, $\mathcal{C}_p$ is compact from $L^2(p)$ to $L^2(p)$ if and only if $\sum_{k=0}^\infty \theta^2_k$ converges. This condition holds if and only if $\mathcal{C}_p$ is bounded from $L^2(p)$ to $L^2(p)$ (see \cite{mat_tolsa_verd}).

\section{A counterexample to Theorem \ref{main_thm} for other kernels}
A natural question is to ask if any analogue Theorem \ref{main_thm} holds also for other singular integral operators of the form
\begin{equation}
Tf(z)=\int_{\mathbb{C}} K(z,w)f(w)d\mu(w),
\end{equation}
where $K$ is a kernel in a proper class and the singular integral operator has to be understood in the usual sense.
For a kernel good enough so that the T(1)-theorem applies, similar considerations as the ones for the sufficiency in the proof of Theorem \ref{main_thm} apply. In particular, in order to have $T$ is compact from $L^2(\mu)$ to $L^2(\mu)$ it suffices to require
\begin{enumerate}
\item $\Theta^*_{\mu}(z)=0$ for every $z\in\mathbb{C}.$\label{d_con}
\item ${||T\chi_Q||^2_{L^2(\mu\llcorner Q)}}/{\mu(Q)}\to 0$ as $l(Q)\to 0.$ 
\item ${||T^*\chi_Q||^2_{L^2(\mu\llcorner Q)}}/{\mu(Q)}\to 0$ as $l(Q)\to 0.$
\end{enumerate}
However, these conditions turn out not to be necessary even in easy cases. An immediate example that shows that the density condition (\ref{d_con}) is not necessary is the operator with kernel
\begin{equation}
K(z,w)=\frac{\Im(z-w)}{|z-w|^2}
\end{equation}
and the measure $\mu=\mathcal{H}^1\llcorner ((0,1)\times \{0\}).$\\
This operator (trivially) belongs to 
$K(L^2(\mu),L^2(\mu))$ even though $\mu$ has positive linear density at each point of $(0,1)\times \{0\}$.
\label{Bibliography}

\end{document}